\documentclass[reqno]{amsart}
\newtheorem{thm}{Theorem}
\newtheorem{lem}[thm]{Lemma}

\theoremstyle{definition}

\providecommand{\abs}[1]{\lvert#1\rvert}
\providecommand{\Abs}[1]{\Bigl\lvert#1\Bigr\rvert}

\begin{document}

\title[Urns with random barriers]{Asymptotics for randomly reinforced urns with random barriers}

\author{Patrizia Berti}
\address{Patrizia Berti, Dipartimento di Scienze Fisiche, Informatiche e Matematiche, Universita' di Modena e Reggio-Emilia, via Campi 213/B, 41100 Modena, Italy}
\email{patrizia.berti@unimore.it}

\author{Irene Crimaldi}
\address{Irene Crimaldi, IMT Institute for Advanced Studies, Piazza San Ponziano 6, 55100 Lucca, Italy}
\email{irene.crimaldi@imtlucca.it}

\author{Luca Pratelli}
\address{Luca Pratelli, Accademia Navale, viale Italia 72, 57100 Livorno,
Italy} \email{pratel@mail.dm.unipi.it}

\author{Pietro Rigo}
\address{Pietro Rigo (corresponding author), Dipartimento di Matematica ``F. Casorati'', Universita' di Pavia, via Ferrata 1, 27100 Pavia, Italy}
\email{pietro.rigo@unipv.it}



\date{\today}

\begin{abstract}
An urn contains black and red balls. Let $Z_n$ be the proportion of black balls at time $n$ and $0\le L<U\le 1$ random barriers. At each time $n$, a ball $b_n$ is drawn. If $b_n$ is black and $Z_{n-1}<U$, then $b_n$ is replaced together with a random number $B_n$ of black balls. If $b_n$ is red and $Z_{n-1}>L$, then $b_n$ is replaced together with a random number $R_n$ of red balls. Otherwise, no additional balls are added, and $b_n$ alone is replaced. In this paper, we assume $R_n=B_n$. Then, under mild conditions, it is shown that $Z_n\overset{a.s.}\longrightarrow Z$ for some random variable $Z$, and
\begin{gather*}
D_n:=\sqrt{n}\,(Z_n-Z)\longrightarrow\mathcal{N}(0,\sigma^2)\quad\text{conditionally a.s.}
\end{gather*}
where $\sigma^2$ is a certain random variance. Almost sure conditional convergence means that
\begin{gather*}
P\bigl(D_n\in\cdot\mid\mathcal{G}_n\bigr)\overset{weakly}\longrightarrow\mathcal{N}(0,\,\sigma^2)\quad\text{a.s.}
\end{gather*}
where $P\bigl(D_n\in\cdot\mid\mathcal{G}_n\bigr)$ is a regular version of the conditional distribution of $D_n$ given the past $\mathcal{G}_n$. Thus, in particular, one obtains $D_n\longrightarrow\mathcal{N}(0,\sigma^2)$ stably. It is also shown that $L<Z<U$ a.s. and $Z$ has non-atomic distribution.
\end{abstract}

\maketitle

\section{Introduction}

In recent times, there is a growing interest on randomly reinforced urns. A meaningful version of the latter, introduced in \cite{AGP} and supported by real applications, is the following.

\subsection{Framework}\label{fr6w} An urn contains $b>0$ black balls and $r>0$ red balls. At each time, a ball is drawn and then replaced, possibly together with a random number of balls of the same color. Precisely, for each $n\ge 1$, let $b_n$ denote the ball drawn at time $n$ and $Z_n$ the proportion of black balls in the urn at time $n$. Then,

\vspace{0.2cm}

\begin{itemize}

\item If $b_n$ is black and $Z_{n-1}<U$, where $U$ is a random barrier, $b_n$ is replaced together with a random number $B_n\ge 0$ of black balls;

\item If $b_n$ is red and $Z_{n-1}>L$, where $L<U$ is another random barrier, $b_n$ is replaced together with a random number $R_n\ge 0$ of red balls;

\item Otherwise, $b_n$ is replaced without additional balls, so that the composition of the urn does not change.

\end{itemize}

\vspace{0.2cm}

To model such urns, we fix a probability space $(\Omega,\mathcal{A},P)$ supporting the random variables $(L,U,X_n,B_n,R_n:n\ge 1)$ such that
\begin{gather*}
0\le L<U\le 1;\quad X_n\in\{0,1\};\quad 0\le B_n,\,R_n\le c\quad\text{for some constant }c.
\end{gather*}
We let
\begin{gather*}
\mathcal{G}_0=\sigma(L,U),\quad\mathcal{G}_n=\sigma\bigl(L,U,X_1,B_1,R_1,\ldots,X_n,B_n,R_n\bigr),\quad Z_0=b/(b+r),
\vspace{0.3cm}\\Z_n=\frac{b+\sum_{i=1}^nX_i\,B_i\,I_{\{Z_{i-1}<U\}}}{b+r+\sum_{i=1}^n\bigl[X_i\,B_i\,I_{\{Z_{i-1}<U\}}+(1-X_i)\,R_i\,I_{\{Z_{i-1}>L\}}\bigr]}.
\end{gather*}
Further, we assume
\begin{gather*}
E(X_{n+1}\mid\mathcal{G}_n)=Z_n\text{ a.s.}\quad\text{and}
\\(B_n,\,R_n)\quad\text{independent of}\quad\sigma\bigl(\mathcal{G}_{n-1},\,X_n).
\end{gather*}

Clearly, $X_n$ should be regarded as the indicator of the event $\{$black ball at time $n\}$ and $Z_n$ as the proportion of black balls in the urn at time $n$.

\subsection{State of the art} Though the literature on randomly reinforced urns is quite huge, random barriers are not so popular. In other terms, the case $L=0$ and $U=1$ is widely investigated (see e.g. \cite{AMS}, \cite{BH}-\cite{BCPR}, \cite{CPS}, \cite{LP}-\cite{ZHCC} and references therein) but
\begin{equation*}
P\bigl(\{L>0\}\cup\{U<1\}\bigr)>0
\end{equation*}
is almost neglected. To our knowledge, the only explicit reference is \cite{AGP}. In such a paper, the barriers $L$ and $U$ are not random (i.e., they are constant) and $(B_n)$ and $(R_n)$ are independent sequences of i.i.d. random variables. Then, it is shown that $Z_n\overset{a.s.}\longrightarrow L$ if $E(B_1)<E(R_1)$ and $Z_n\overset{a.s.}\longrightarrow U$ if $E(B_1)>E(R_1)$. Moreover, in case $E(B_1)=E(R_1)$, a method for estimating the unknown mean is proposed. The method is based on $Z_n\overset{a.s.}\longrightarrow Z$, where the random variable $Z$ has non-atomic distribution, if $E(B_1)=E(R_1)$. However, convergence of $Z_n$ when $E(B_1)=E(R_1)$ is stated without a proof. In addition, the limiting distribution of $Z_n-Z$ is not investigated at all.

\subsection{Results} In a sense, this paper deals with the opposite case with respect to \cite{AGP}. Indeed, while $(B_n)$ and $(R_n)$ are independent sequences in \cite{AGP}, throughout this paper it is assumed that
\begin{equation*}
R_n=B_n\quad\text{for each }n\ge 1.
\end{equation*}
Under this assumption, the following two results are proved.

\begin{thm}\label{t1}
In the framework of Subsection \ref{fr6w}, suppose
\begin{gather*}
R_n=B_n\quad\text{and}\quad\liminf_nE(B_n)>0.
\end{gather*}
Then,
\begin{gather*}
Z_n\overset{a.s.}\longrightarrow Z
\end{gather*}
for some random variable $Z$ such that $L\le Z\le U$ and $0<Z<1$ a.s.
\end{thm}

\begin{thm}\label{t2}
In the framework of Subsection \ref{fr6w}, suppose
\begin{gather*}
R_n=B_n,\quad m:=\lim_nE(B_n)>0\quad\text{and}\quad q:=\lim_nE(B_n^2).
\end{gather*}
Define
\begin{gather*}
D_n=\sqrt{n}\,(Z_n-Z)\quad\text{and}\quad\sigma^2=q\,Z\,(1-Z)/m^2,
\end{gather*}
where $Z$ is the a.s. limit of $Z_n$. Then,
\begin{gather*}
D_n\longrightarrow\mathcal{N}(0,\,\sigma^2)\quad\text{conditionally a.s. with respect to }(\mathcal{G}_n).
\end{gather*}
Moreover, $Z$ has a non-atomic distribution and $L<Z<U$ a.s.
\end{thm}

In Theorem \ref{t2}, $\mathcal{N}(a,b)$ denotes the Gaussian law with
mean $a$ and variance $b\geq 0$, where $\mathcal{N}(a,0)=\delta_a$. Almost sure conditional convergence is a strong form of stable convergence, introduced in \cite{CLP}-\cite{CRIM} and involved in \cite{AMS}, \cite{BCPR}, \cite{ZHANG}, \cite{ZHCC}. The general definition is discussed in Section \ref{ascconv}. In the present case, it means that
\begin{gather*}
P\bigl(D_n\in\cdot\mid\mathcal{G}_n\bigr)(\omega)\overset{weakly}\longrightarrow\mathcal{N}(0,\,\sigma^2(\omega))\quad\text{for almost all }\omega\in\Omega
\end{gather*}
where $P\bigl(D_n\in\cdot\mid\mathcal{G}_n\bigr)$ is a regular version of the conditional distribution of $D_n$ given $\mathcal{G}_n$. Thus, in particular, Theorem \ref{t2} yields
\begin{gather*}
D_n\longrightarrow\mathcal{N}(0,\,\sigma^2)\quad\text{stably};
\end{gather*}
see Lemma \ref{m6c8}.

Theorems \ref{t1}-\ref{t2} establish the asymptotics for randomly reinforced urns with random barriers when $R_n=B_n$. The case $R_n\ne B_n$, as well as some other possible developments, are discussed in Section \ref{g6r4e}. A last note is that Theorem \ref{t2} agrees with the result obtained when random barriers are not taken into account. Indeed, if $L=0$ and $U=1$, Theorem \ref{t2} follows from \cite[Corollary 3]{BCPR}.

\section{Almost sure conditional convergence}\label{ascconv}

Almost sure conditional convergence, introduced in \cite{CLP}-\cite{CRIM}, may be regarded as a strong form of stable convergence. We now make it precise.

Let $(\Omega,\mathcal{A},P)$ be a probability space and $S$ a metric
space. A {\em kernel} on $S$ (or a {\em random probability measure}
on $S$) is a measurable collection $N=\{N(\omega):\omega\in\Omega\}$
of probability measures on the Borel $\sigma$-field on $S$.
Measurability means that
\begin{equation*}
N(\cdot)(f)=\int f(x)\,N(\cdot)(dx)
\end{equation*}
is a real random variable for each bounded Borel map $f:S\rightarrow\mathbb{R}$. To denote such random variable, in the sequel, we will often write $N(f)$ instead of $N(\cdot)(f)$.

For each $n\ge 1$, fix a sub-$\sigma$-field $\mathcal{F}_n\subset\mathcal{A}$. Also, let $(Y_n)$ be a sequence of $S$-valued random variables and $N$ a
kernel on $S$. Say that $Y_n$ converges to $N$, {\it conditionally a.s. with respect to} $(\mathcal{F}_n)$, if
\begin{gather}\label{sc}
E\bigl\{f(Y_n)\mid\mathcal{F}_n\bigr\}\overset{a.s.}\longrightarrow N(f)\quad\text{for each }f\in C_b(S).
\end{gather}
If $S$ is Polish, condition \eqref{sc} has a quite transparent meaning. Suppose in fact $S$ is Polish and fix a regular version $P\bigl(Y_n\in\cdot\mid\mathcal{F}_n\bigr)$ of the conditional distribution of $Y_n$ given $\mathcal{F}_n$. Then, condition \eqref{sc} is equivalent to
\begin{gather*}
P\bigl(Y_n\in\cdot\mid\mathcal{F}_n\bigr)(\omega)\overset{weakly}\longrightarrow N(\omega)\quad\text{for almost all }\omega\in\Omega.
\end{gather*}

So far, $(\mathcal{F}_n)$ is an arbitrary sequence of sub-$\sigma$-fields. Suppose now that $(\mathcal{F}_n)$ is a filtration, in the sense that $\mathcal{F}_n\subset\mathcal{F}_{n+1}\subset\mathcal{A}$ for each $n$. Then, under a mild measurability condition, almost sure conditional convergence implies stable convergence. This is noted in \cite[Section 5]{CLP} but we give a proof to make the paper self-contained. Let
\begin{gather*}
\mathcal{F}_\infty=\sigma\bigl(\cup_n\mathcal{F}_n\bigr).
\end{gather*}

\begin{lem}\label{m6c8}
Suppose $(\mathcal{F}_n)$ is a filtration such that
\begin{gather*}
N(f)\text{ and }Y_n\text{ are }\mathcal{F}_\infty\text{-measurable for all }f\in C_b(S)\text{ and }n\ge 1.
\end{gather*}
If $Y_n\rightarrow N$ conditionally a.s. with respect to $(\mathcal{F}_n)$, then $Y_n\rightarrow N$ stably, that is
\begin{gather*}
E\bigl\{N(f)\mid H\bigr\}=\lim_nE\bigl\{f(Y_n)\mid H\bigr\}
\end{gather*}
whenever $f\in C_b(S)$, $H\in\mathcal{A}$ and $P(H)>0$.
\end{lem}

\begin{proof}
Let $f\in C_b(S)$ and $H\in\mathcal{A}$. If $H\in\cup_n\mathcal{F}_n$, then $H\in\mathcal{F}_n$ for each sufficiently large $n$, so that
\begin{gather*}
E\bigl\{N(f)\,I_H\bigr\}=\lim_nE\Bigl(E\bigl\{f(Y_n)\mid\mathcal{F}_n\bigr\}\,I_H\Bigr)=\lim_nE\bigl\{f(Y_n)\,I_H\}.
\end{gather*}
Since $\cup_n\mathcal{F}_n$ is a field, by standard arguments one obtains
\begin{gather*}
E\bigl\{N(f)\,V\bigr\}=\lim_nE\bigl\{f(Y_n)\,V\}\quad\text{whenever }V\text{ is bounded and }\mathcal{F}_\infty\text{-measurable}.
\end{gather*}
Hence, for arbitrary $H\in\mathcal{A}$, the measurability condition implies
\begin{gather*}
E\bigl\{N(f)\,I_H\bigr\}=E\Bigl(N(f)\,E(I_H\mid\mathcal{F}_\infty)\Bigr)=\lim_nE\Bigl(f(Y_n)\,E(I_H\mid\mathcal{F}_\infty)\Bigr)=\lim_nE\bigl\{f(Y_n)\,I_H\}.
\end{gather*}
\end{proof}
Note that the measurability condition of Lemma \ref{m6c8} is trivially true if $\mathcal{F}_\infty=\mathcal{A}$.

We refer to \cite{CLP}-\cite{CRIM} for more on almost sure conditional convergence. Here, for easy of reference, we report three useful facts. The first and the second are already known (see \cite[Proposition 1 and Lemma 2]{BCPR} and \cite[Theorem 2.2]{CRIM}) while the third is a quick consequence of condition \eqref{sc}. In each of these facts, $(\mathcal{F}_n)$ is a filtration.

\begin{lem}\label{f78j2nhj7t}
Suppose the $Y_n$ are real random variables such that $Y_n\overset{a.s.}\longrightarrow Y$. Then,
\begin{gather*}
\sqrt{n}\,(Y_n-Y)\longrightarrow\mathcal{N}(0,U),\quad\text{conditionally a.s. with respect to }(\mathcal{F}_n),
\end{gather*}
where $U$ is a real random variable, provided

\vspace{0.2cm}

\begin{itemize}

\item[(i)] $(Y_n)$ is a uniformly integrable martingale with respect to $(\mathcal{F}_n)$;

\vspace{0.2cm}

\item[(ii)] $E\bigl\{\sup_n\sqrt{n}\,\abs{Y_n-Y_{n-1}}\bigr\}<\infty$;

\vspace{0.2cm}

\item[(iii)] $n\sum_{k\ge n}(Y_k-Y_{k-1})^2\overset{a.s.}\longrightarrow U$.

\end{itemize}

\end{lem}

\begin{lem}\label{cx3ej876zq1}
Suppose the $Y_n$ are real random variables. If $(Y_n)$ is adapted to $(\mathcal{F}_n)$, $\sum_n n^{-2}E(Y_n^2)<\infty$ and
$E\bigl(Y_{n+1}\mid\mathcal{F}_n\bigr)\overset{a.s.}\longrightarrow
U$, for some real random variable $U$, then
\begin{equation*}
n\sum_{k\ge n}\frac{Y_k}{k^2}\,\overset{a.s.}\longrightarrow
U\quad\text{and}\quad\frac{1}{n}\sum_{k=1}^nY_k\overset{a.s.}\longrightarrow
U.
\end{equation*}
\end{lem}

\begin{lem}\label{k9cxs34}
Suppose $Y_n\rightarrow N$ conditionally a.s. with respect to $(\mathcal{F}_n)$. Define $Q(A)=E\bigl\{I_A\,V\}$ for $A\in\mathcal{A}$, where $V\ge 0$, $E(V)=1$ and $V$ is $\mathcal{F}_\infty$-measurable. Then $Y_n\rightarrow N$, conditionally a.s. with respect to $(\mathcal{F}_n)$, under $Q$ as well.
\end{lem}

\begin{proof}
Suppose first $\sup V<\infty$ and define $K_n=V-E(V\mid\mathcal{F}_n)$. Given $f\in C_b(S)$,
\begin{gather*}
E_Q\bigl\{f(Y_n)\mid\mathcal{F}_n\bigr\}=\frac{E\bigl\{V\,f(Y_n)\mid\mathcal{F}_n\bigr\}}{E(V\mid\mathcal{F}_n)}
=E\bigl\{f(Y_n)\mid\mathcal{F}_n\bigr\}+\frac{E\bigl\{K_n\,f(Y_n)\mid\mathcal{F}_n\bigr\}}{E(V\mid\mathcal{F}_n)},\quad Q\text{-a.s.,}
\end{gather*}
where $E_Q$ denotes expectation under $Q$. Since $\sigma(V)\subset\mathcal{F}_\infty$ and $\abs{K_n}\le\sup V$ a.s., the martingale convergence theorem (in the version of \cite{BD}) implies
\begin{gather*}
E(V\mid\mathcal{F}_n)\overset{a.s.}\longrightarrow E(V\mid\mathcal{F}_\infty)=V\quad\text{and}
\\\Abs{\,E\bigl\{K_n\,f(Y_n)\mid\mathcal{F}_n\bigr\}}\le\sup\abs{f}\,E\bigl\{\abs{K_n}\mid\mathcal{F}_n\bigr\}\overset{a.s.}\longrightarrow 0.
\end{gather*}
Since $Q(V>0)=1$, one obtains $E_Q\bigl\{f(Y_n)\mid\mathcal{F}_n\bigr\}\rightarrow N(f)$, $Q$-a.s. This concludes the proof for bounded $V$. If $V$ is not bounded, it suffices to reply $V$ with $V\,I_{\{V\le v\}}/E\bigl(V\,I_{\{V\le v\}}\bigr)$ and to take the limit as $v\rightarrow\infty$.
\end{proof}

\section{Proofs}

In the sequel, for any events $A_n\in\mathcal{A}$ and $B\in\mathcal{A}$, we say that $A_n$ {\em is eventually true on } $B$ (or, more briefly, $A_n$ {\em eventually on } $B$) whenever
\begin{gather*}
P\bigl\{\omega\in B:\omega\notin A_n\text{ for infinitely many }n\bigr\}=0.
\end{gather*}

Assume the conditions of Subsection \ref{fr6w} and $R_n=B_n$. Let
\begin{gather*}
S_n=b+r+\sum_{i=1}^n\bigl[X_i\,B_i\,I_{\{Z_{i-1}<U\}}+(1-X_i)\,B_i\,I_{\{Z_{i-1}>L\}}\bigr]
\end{gather*}
denote the denominator of $Z_n$, namely, the number of balls in the urn at time $n$. Also, the filtration $(\mathcal{G}_n)$ is abbreviated by $\mathcal{G}$.

After some (tedious but easy) algebra, one obtains
\begin{gather*}
Z_{n+1}-Z_n=Z_n\,H_n+\Delta_{n+1},
\end{gather*}
where
\begin{gather*}
H_n=\frac{B_{n+1}}{S_n+B_{n+1}}\,(1-Z_n)\,\Bigl(I_{\{Z_n<U\}}-I_{\{Z_n>L\}}\Bigr),
\\\Delta_{n+1}=\frac{B_{n+1}}{S_n+B_{n+1}}\,(X_{n+1}-Z_n)\,\Bigl((1-Z_n)I_{\{Z_n<U\}}+Z_nI_{\{Z_n>L\}}\Bigr).
\end{gather*}
This writing of $Z_{n+1}-Z_n$ is fundamental for our purposes.

\subsection{Proof of Theorem \ref{t1}}\label{u6t54e}
In this subsection, it is assumed that
\begin{gather*}
\liminf_nE(B_n)>0.
\end{gather*}

Since
\begin{gather*}
E\bigl(X_{n+1}\mid\mathcal{G}_n,\,B_{n+1}\bigr)=E(X_{n+1}\mid\mathcal{G}_n)=Z_n\quad\text{a.s.},
\end{gather*}
then
\begin{gather*}
E(\Delta_{n+1}\mid\mathcal{G}_n)=0\quad\text{a.s.}
\end{gather*}
This fact has two useful consequences. First,
\begin{gather*}
M_n=\sum_{i=1}^n\Delta_i
\end{gather*}
is a $\mathcal{G}$-martingale. Second, $(Z_n)$ is a $\mathcal{G}$-sub-martingale in case $U=1$. In fact, $U=1$ implies $H_n\ge 0$, so that
\begin{gather*}
E\bigl\{Z_{n+1}-Z_n\mid\mathcal{G}_n\bigr\}=E\bigl\{Z_nH_n+\Delta_{n+1}\mid\mathcal{G}_n\bigr\}=Z_n\,E(H_n\mid\mathcal{G}_n)\ge 0\quad\text{a.s.}
\end{gather*}
Similarly, if $L=0$ then $(Z_n)$ is a $\mathcal{G}$-super-martingale. Therefore, it is not hard to see that $Z_n$ converges a.s. on the set $\{L=0\}\cup\{U=1\}$.

We next state two lemmas.

\begin{lem}\label{b76ty8k9} Let $Z_*=\liminf_nZ_n$ and $Z^*=\limsup_nZ_n$. Each of the following statements implies the subsequent:

\vspace{0.2cm}

\begin{itemize}

\item[(a)] $0<L<U<1$ a.s.;

\vspace{0.2cm}

\item[(b)] $0<Z_*\le Z^*<1$ a.s.;

\vspace{0.2cm}

\item[(c)] $\liminf_n(S_n/n)>0$ a.s.;

\vspace{0.2cm}

\item[(d)] $M_n$ converges a.s.

\end{itemize}

\vspace{0.2cm}

\end{lem}

\begin{proof} {\bf ``(a) $\Longrightarrow$ (b)''.} Let $H=\{Z_*=0,\,L>0\}$. On the set $H$, one obtains
\begin{gather*}
\sup_nS_n=\infty,\quad\lim_n(Z_{n+1}-Z_n)=0,\quad Z_n>L\text{ for infinitely many }n.
\end{gather*}
Define $\tau_0=0$ and
\begin{gather*}
\tau_n=\inf\bigl\{k:k>\tau_{n-1},\quad Z_{k-1}>L,\quad Z_k\le L\bigr\}.
\end{gather*}
Then, $\tau_n<\infty$ for all $n$ on $H$. Observe now that $Z_j\ge Z_{j-1}$ whenever $Z_{j-1}\le L$. Hence, $Z_*=\liminf_nZ_n=\liminf_nZ_{\tau_n}$ on $H$, which implies the contradiction
\begin{gather*}
Z_*\ge\liminf_nZ_{\tau_{n-1}}+\liminf_n(Z_{\tau_n}-Z_{\tau_{n-1}})=\liminf_nZ_{\tau_{n-1}}\ge L>0\quad\text{a.s. on }H.
\end{gather*}
Thus, under (a), one obtains $P(Z_*=0)=P(H)=0$. Similarly, $P(Z^*=1)=0$.

{\bf ``(b) $\Longrightarrow$ (c)''.} Define
\begin{gather*}
K_n=\sum_{i=1}^n\,\frac{I_{\{Z_{i-1}<U\}}\bigl[X_iB_i-Z_{i-1}E(B_i)\bigr]+I_{\{Z_{i-1}>L\}}\bigl[(1-X_i)B_i-(1-Z_{i-1})E(B_i)\bigr]}{i}.
\end{gather*}
Since $K_n$ is a $\mathcal{G}$-martingale and $\sup_nE(K_n^2)<\infty$, then $K_n$ converges a.s. Thus, Kronecker lemma implies $(1/n)\,\sum_{i=1}^ni\,K_i\overset{a.s.}\longrightarrow 0$, so that
\begin{gather*}
\liminf_n\frac{S_n}{n}=\liminf_n\,\frac{1}{n}\,\sum_{i=1}^n\Bigl(I_{\{Z_{i-1}<U\}}Z_{i-1}E(B_i)+I_{\{Z_{i-1}>L\}}\,(1-Z_{i-1})E(B_i)\Bigr)\text{ a.s.}
\end{gather*}
Since $I_{\{Z_{i-1}<U\}}+I_{\{Z_{i-1}>L\}}\ge 1$, one finally obtains
\begin{gather*}
\liminf_n\frac{S_n}{n}\ge \bigl\{Z_*\wedge (1-Z^*)\bigr\}\,\liminf_nE(B_n)>0\quad\text{a.s.}
\end{gather*}

{\bf ``(c) $\Longrightarrow$ (d)''.} Since $0\le B_{n+1}\le c$,
\begin{gather*}
E\bigl\{(M_{n+1}-M_n)^2\mid\mathcal{G}_n\bigr\}=E(\Delta_{n+1}^2\mid\mathcal{G}_n)\le E\Bigl(\frac{B_{n+1}^2}{S_n^2}\mid\mathcal{G}_n\Bigr)=\frac{E(B_{n+1}^2)}{S_n^2}\le\frac{c^2}{n^2}\,\frac{1}{(S_n/n)^2}\quad\text{a.s.}
\end{gather*}
Thus, $\sum_nE\bigl\{(M_{n+1}-M_n)^2\mid\mathcal{G}_n\bigr\}<\infty$ a.s. by condition (c). It follows that the $\mathcal{G}$-martingale $M_n$ converges a.s.
\end{proof}

\begin{lem}\label{u8h5rd3}
If $\,\liminf_n(S_n/n)>0$ a.s., then $P(D)=0$ where
\begin{gather*}
D=\bigl\{Z_a\le L\text{ for infinitely many }a\text{ and }Z_b\ge U\text{ for infinitely many }b\bigr\}.
\end{gather*}
\end{lem}

\begin{proof}
On $D$, there is a sequence $(a_n,b_n)$ such that $a_1<b_1<a_2<b_2<\ldots$ and
\begin{gather*}
Z_{a_n}\le L,\quad Z_{b_n}\ge U,\quad L<Z_k<U\text{ for each }a_n<k<b_n.
\end{gather*}
Since $H_k=0$ if $a_n<k<b_n$, then
\begin{gather*}
U-L\le Z_{b_n}-Z_{a_n}=\sum_{k=a_n}^{b_n-1}(Z_{k+1}-Z_k)=\sum_{k=a_n}^{b_n-1}(Z_kH_k+\Delta_{k+1})=Z_{a_n}H_{a_n}+M_{b_n}-M_{a_n}.
\end{gather*}
Since $\liminf_n(S_n/n)>0$ a.s., then $\sup_nS_n=\infty$ a.s., which implies $H_n\overset{a.s.}\longrightarrow 0$. Also, by Lemma \ref{b76ty8k9}, $M_n$ converges a.s. Hence, taking the limit as $n\rightarrow\infty$, one obtains $U-L\le 0$ a.s. on $D$. Therefore, $P(D)=0$.
\end{proof}

We are now ready to prove a.s. convergence of $Z_n$.

Since $Z_n$ converges a.s. on the set $\{L=0\}\cup\{U=1\}$, it can be assumed $P(0<L<U<1)>0$. In turn, up to replacing $P$ with $P(\cdot\mid 0<L<U<1)$, it can be assumed $0<L<U<1$ a.s. Then, Lemmas \ref{b76ty8k9}-\ref{u8h5rd3} imply $P(D^c)=1$ and a.s. convergence of $M_n$. Write
\begin{gather*}
Z_n-Z_0=\sum_{i=0}^{n-1}\,(Z_{i+1}-Z_i)=\sum_{i=0}^{n-1}Z_iH_i+M_n=K_n+M_n
\end{gather*}
where $K_n=\sum_{i=0}^{n-1}Z_iH_i$. On the set $D^c$, one has either $Z_iH_i\ge 0$ eventually or $Z_iH_i\le 0$ eventually. Hence, on $D^c$, the sequence $K_n$ converges if and only if it is bounded. But $K_n$ is a.s. bounded, since $\abs{K_n}\le 1+\sup_k\abs{M_k}$ and $M_n$ converges a.s. Thus, $Z_n$ converges a.s. on $D^c$. This proves a.s. convergence of $Z_n$ for $P(D^c)=1$.

Let $Z$ denote the a.s. limit of $Z_n$. Since $Z_n\overset{a.s.}\longrightarrow 1$ on the set $\{Z<L\}$ and $Z_n\overset{a.s.}\longrightarrow 0$ on the set $\{Z>U\}$, then $L\le Z\le U$ a.s.

It remains to see that
\begin{gather*}
P(Z=0)=P(Z=1)=0.
\end{gather*}
We just prove $P(Z=1)=0$. The proof of $P(Z=0)=0$ is quite analogous.

Since $Z\le U\le 1$ a.s., then $P(Z=1)\le P(U=1)$. Thus, it can be assumed $P(U=1)>0$. In turn, up to replacing $P$ with $P(\cdot\mid U=1)$, it can be assumed $U=1$ everywhere. Then, $Z_n$ is a $\mathcal{G}$-sub-martingale, so that
\begin{gather*}
Y_n=Z_n/(1-Z_n)
\end{gather*}
is still a $\mathcal{G}$-sub-martingale. Let $H=\bigl\{\,\sum_nE\bigl\{Y_{n+1}-Y_n\mid\mathcal{G}_n\bigr\}<\infty\bigr\}$. Since $Y_n$ is a positive $\mathcal{G}$-sub-martingale, $Y_n$ converges a.s. (to a real random variable) on the set $H$. Thus, to get $P(Z=1)=0$, it suffices to show that
\begin{gather}\label{ne4veu9}
\sum_nE\bigl\{Y_{n+1}-Y_n\mid\mathcal{G}_n\bigr\}<\infty\quad\text{a.s. on the set }\{Z=1\}.
\end{gather}

To prove \eqref{ne4veu9}, let
\begin{gather*}
J_n=b+\sum_{i=1}^nX_i\,B_i\quad\text{and}\quad L_n=S_n-J_n=r+\sum_{i=1}^n(1-X_i)\,B_i\,I_{\{Z_{i-1}>L\}}
\end{gather*}
be the numbers of black balls and red balls, respectively, in the urn at time $n$ (recall that $U=1$, so that $Z_{i-1}<U$ is automatically true). On noting that $Y_n=J_n/L_n$, one obtains
\begin{gather*}
E\bigl\{Y_{n+1}-Y_n\mid\mathcal{G}_n\bigr\}=-Y_n+E\Bigl\{\frac{J_n+B_{n+1}}{L_n}\,X_{n+1}\,+\,\frac{J_n}{L_n+I_{\{Z_n>L\}}B_{n+1}}\,(1-X_{n+1})\mid\mathcal{G}_n\Bigr\}
\\=-Y_n+Z_n\,E\Bigl\{\frac{J_n+B_{n+1}}{L_n}\mid\mathcal{G}_n\Bigr\}+(1-Z_n)\,E\Bigl\{\frac{J_n}{L_n+I_{\{Z_n>L\}}B_{n+1}}\mid\mathcal{G}_n\Bigr\}
\\=-Y_n\,(1-Z_n)+\frac{Z_n\,E(B_{n+1})}{L_n}+Y_n\,(1-Z_n)\,E\Bigl\{\frac{L_n}{L_n+I_{\{Z_n>L\}}B_{n+1}}\mid\mathcal{G}_n\Bigr\}
\\=\frac{Z_n\,E(B_{n+1})}{L_n}-Z_n\,I_{\{Z_n>L\}}\,E\Bigl\{\frac{B_{n+1}}{L_n+I_{\{Z_n>L\}}B_{n+1}}\mid\mathcal{G}_n\Bigr\}
\\\le\frac{Z_n\,E(B_{n+1})}{L_n}-Z_n\,I_{\{Z_n>L\}}\,E\Bigl\{\frac{B_{n+1}}{L_n+c}\mid\mathcal{G}_n\Bigr\}
\\=\frac{Z_n\,E(B_{n+1})}{L_n}-Z_n\,I_{\{Z_n>L\}}\,\frac{E(B_{n+1})}{L_n+c}\quad\text{a.s.}
\end{gather*}
Since $Z_n\overset{a.s.}\longrightarrow Z$, then $Z_n>L$ eventually on the set $\{Z=1\}$. Hence,
\begin{gather*}
E\bigl\{Y_{n+1}-Y_n\mid\mathcal{G}_n\bigr\}\le Z_n\,E(B_{n+1})\,\Bigl(\frac{1}{L_n}-\frac{1}{L_n+c}\Bigr)\le\frac{c^2}{L_n^2}\quad\text{eventually on }\{Z=1\}.
\end{gather*}

Next, given $k\in (1,2)$, it is not hard to see that
\begin{gather*}
E\Bigl\{\frac{J_{n+1}}{L_{n+1}^k}-\frac{J_n}{L_n^k}\mid\mathcal{G}_n\Bigr\}\le 0\quad\text{eventually on }\{Z=1\}.
\end{gather*}
We omit the calculations for they exactly agree with those for proving \cite[Lemma A.1(ii)]{MF}. Thus, the sequence $J_n/L_n^k$ converges a.s. on $\{Z=1\}$. Furthermore, independence of the $B_n$ yields
\begin{gather*}
\liminf_n\frac{J_n}{n}=\liminf_n\frac{\sum_{i=1}^nB_i}{n}\,\frac{S_n}{\sum_{i=1}^nB_i}\,Z_n=\liminf_n\frac{\sum_{i=1}^nB_i}{n}
\\=\liminf_n\frac{\sum_{i=1}^nE(B_i)}{n}\ge\liminf_nE(B_n)>0\quad\text{a.s. on }\{Z=1\}.
\end{gather*}
Given any $\gamma<1$, it follows that
\begin{gather*}
\frac{n^\gamma}{L_n^k}=\frac{n^\gamma}{n}\,\frac{n}{J_n}\,\frac{J_n}{L_n^k}\overset{a.s.}\longrightarrow 0\,\text{ a.s. on }\,\{Z=1\}.
\end{gather*}
Thus,
\begin{gather*}
L_n>n^{\gamma/k}\,\text{ eventually on }\,\{Z=1\}.
\end{gather*}
Since $k<2$, one can take $\gamma<1$ such that $\gamma/k>1/2$. Therefore, condition \eqref{ne4veu9} holds, and this concludes the proof of Theorem \ref{t1}.

\vspace{0.4cm}

\subsection{Proof of Theorem \ref{t2}}
In this subsection, it is assumed that
\begin{gather*}
m:=\lim_nE(B_n)>0\quad\text{and}\quad q:=\lim_nE(B_n^2).
\end{gather*}
By Theorem \ref{t1}, $Z_n\overset{a.s.}\longrightarrow Z$ for some random variable $Z$ such that $L\le Z\le U$ and $0<Z<1$ a.s.

On noting that
 \begin{gather*}
0<Z_*=Z=Z^*<1\quad\text{a.s.,}
\end{gather*}
the same argument used after Lemma \ref{u8h5rd3} yields $\sum_nZ_n\abs{H_n}<\infty$ a.s. Since $Z>0$ a.s., it follows that
\begin{gather}\label{ur69kh7}
\sum_n\abs{H_n}<\infty\quad\text{a.s.}
\end{gather}
Define
\begin{gather*}
T_n=\prod_{i=1}^{n-1}(1+H_i)\quad\text{and}\quad W_n=\frac{Z_n}{T_n}.
\end{gather*}
Condition \eqref{ur69kh7} implies $T_n\overset{a.s.}\longrightarrow T$, for some real random variable $T>0$, so that
\begin{gather*}
W_n\overset{a.s.}\longrightarrow Z/T:=W.
\end{gather*}

Our next goal is to show that $\sqrt{n}\,(W_n-W)$ converges conditionally a.s. To this end, we first fix the asymptotic behavior of $S_n$.

\begin{lem} $S_n/n\overset{a.s.}\longrightarrow m$.
\end{lem}

\begin{proof}
Let $Q_n=\sum_{i=1}^nB_i\,\bigl[X_iI_{\{Z_{i-1}\ge U\}}+(1-X_i)I_{\{Z_{i-1}\le L\}}\bigr]$. Since $Z_i\overset{a.s.}\longrightarrow Z$, then $1-Z_i>(1-Z)/2$ eventually. Moreover, $I_{\{Z_{i-1}\ge U\}}+I_{\{Z_{i-1}\le L\}}=\abs{I_{\{Z_{i-1}<U\}}-I_{\{Z_{i-1}>L\}}}$. Therefore,
\begin{gather*}
B_i\,\bigl[X_iI_{\{Z_{i-1}\ge U\}}+(1-X_i)I_{\{Z_{i-1}\le L\}}\bigr]\le B_i\,\abs{I_{\{Z_{i-1}< U\}}-I_{\{Z_{i-1}> L\}}}
\\=\abs{H_{i-1}}\,\frac{S_{i-1}+B_i}{1-Z_{i-1}}\le\frac{2\,\abs{H_{i-1}}}{1-Z}\,(S_{i-1}+B_i)\quad\text{eventually.}
\end{gather*}
By condition \eqref{ur69kh7} and Kronecker lemma,
\begin{gather*}
\frac{Q_n}{S_n}\le\frac{2}{1-Z}\,\frac{1}{S_n}\,\sum_{i=1}^n\abs{H_{i-1}}S_{i-1}\,+\,\frac{2\,c}{1-Z}\,\frac{1}{S_n}\,\sum_{i=1}^n\abs{H_{i-1}}\overset{a.s.}\longrightarrow 0.
\end{gather*}
Hence,
\begin{gather*}
\frac{Q_n}{n}=\frac{Q_n}{S_n}\,\frac{S_n}{n}\le\frac{Q_n}{S_n}\,\,\frac{r+b+nc}{n}\overset{a.s.}\longrightarrow 0.
\end{gather*}
On noting that $(1/n)\,\sum_{i=1}^nB_i\overset{a.s.}\longrightarrow m$, one finally obtains
\begin{gather*}
\frac{S_n}{n}=\frac{r+b}{n}-\frac{Q_n}{n}+\frac{\sum_{i=1}^nB_i}{n}\overset{a.s.}\longrightarrow m.
\end{gather*}
\end{proof}

In view of the next lemma, we recall that
\begin{gather*}
\sigma^2=q\,Z\,(1-Z)/m^2.
\end{gather*}

\begin{lem}\label{pierc3l9}
\begin{gather*}
\sqrt{n}\,(W_n-W)\longrightarrow\mathcal{N}\bigl(0,\,\sigma^2/T^2\bigr)\quad\text{conditionally a.s. with respect to }\mathcal{G}.
\end{gather*}
\end{lem}

\begin{proof}
First note that $W_n$ can be written as
\begin{gather*}
W_n=Z_1+\sum_{i=1}^{n-1}\frac{\Delta_{i+1}}{T_{i+1}}.
\end{gather*}
Thus, $W_n$ is a $\mathcal{G}$-martingale and the obvious strategy would be applying Lemma \ref{f78j2nhj7t} to $Y_n=W_n$. However, conditions (i)-(ii)-(iii) are not easy to check with $Y_n=W_n$. Accordingly, we adopt an approximation procedure.

Given $\epsilon>0$, define
\begin{gather*}
W_n^{(\epsilon)}=Z_1+\sum_{i=1}^{n-1}\frac{\Delta_{i+1}\,I_{A_i}}{\epsilon\vee T_{i+1}}\quad\text{where }A_i=\{2\,S_i>i\,m\}.
\end{gather*}
For fixed $\epsilon>0$, $W_n^{(\epsilon)}$ is still a $\mathcal{G}$-martingale and
\begin{gather*}
\sup_nE\bigl\{(W_n^{(\epsilon)})^2\bigr\}\le 1+\frac{1}{\epsilon^2}\,\sum_{i=1}^\infty E\bigl\{\Delta_{i+1}^2\,I_{A_i}\bigr\}\le 1+\frac{c^2}{\epsilon^2}\,\sum_{i=1}^\infty E\Bigl\{\frac{I_{A_i}}{S_i^2}\Bigr\}
\le 1+\Bigl(\frac{2\,c}{m\,\epsilon}\Bigr)^2\,\sum_{i=1}^\infty\frac{1}{i^2}.
\end{gather*}
Hence, $W_n^{(\epsilon)}\overset{a.s.}\longrightarrow W^{(\epsilon)}$ for some random variable $W^{(\epsilon)}$. Since $S_n/n\overset{a.s.}\longrightarrow m
$, the events $A_n$ are eventually true, so that
\begin{gather*}
W-W_n=\sum_{i\ge n}\frac{\Delta_{i+1}}{T_{i+1}}=\sum_{i\ge n}\frac{\Delta_{i+1}\,I_{A_i}}{\epsilon\vee T_{i+1}}=W^{(\epsilon)}-W_n^{(\epsilon)}\quad\text{eventually on }\{T>\epsilon\}.
\end{gather*}
Therefore, it suffices to show that, for fixed $\epsilon>0$,
\begin{gather*}
\sqrt{n}\,\bigl\{W_n^{(\epsilon)}-W^{(\epsilon)}\bigr\}\longrightarrow\mathcal{N}\bigl(0,\,\sigma^2/(\epsilon\vee T)^2\bigr)\quad\text{conditionally a.s. with respect to }\mathcal{G}.
\end{gather*}
In turn, since  $W_n^{(\epsilon)}$ is a uniformly integrable $\mathcal{G}$-martingale, it suffices to check conditions (ii)-(iii) of Lemma \ref{f78j2nhj7t} with $Y_n=W_n^{(\epsilon)}$ and $U=\sigma^2/(\epsilon\vee T)^2$. As to (ii),
\begin{gather*}
E\Bigl\{\bigl(\sup_n\sqrt{n}\,\abs{W_n^{(\epsilon)}-W_{n-1}^{(\epsilon)}}\bigr)^4\Bigr\}\le\sum_nn^2E\bigl\{(W_n^{(\epsilon)}-W_{n-1}^{(\epsilon)})^4\bigr\}
\\\le\frac{1}{\epsilon^4}\,\sum_nn^2E\bigl\{\Delta_n^4\,I_{A_{n-1}}\bigr\}\le\frac{c^4}{\epsilon^4}\,\sum_nn^2E\Bigl\{\frac{I_{A_{n-1}}}{S_{n-1}^4}\Bigr\}\le\Bigl(\frac{2\,c}{m\,\epsilon}\Bigr)^4\sum_n\frac{n^2}{(n-1)^4}<\infty.
\end{gather*}

We next turn to condition (iii). We have to prove
\begin{gather*}
n\sum_{k\ge n}\bigl(W_k^{(\epsilon)}-W_{k-1}^{(\epsilon)}\big)^2=n\sum_{k\ge n}\frac{I_{A_{k-1}}\,\Delta_k^2}{(\epsilon\vee T_k)^2}\overset{a.s.}\longrightarrow\sigma^2/(\epsilon\vee T)^2.
\end{gather*}
Since $T_k\overset{a.s.}\longrightarrow T$, the above condition reduces to
\begin{gather}\label{capp8uc4}
n\sum_{k\ge n}I_{A_{k-1}}\,\Delta_k^2\overset{a.s.}\longrightarrow\sigma^2.
\end{gather}
Since $1-Z_k>(1-Z)/2$ eventually and $\sum_k\abs{H_k}<\infty$ a.s., Abel summation formula yields
\begin{gather*}
n\sum_{k\ge n}I_{A_{k-1}}\,\frac{B_k^2}{(S_{k-1}+B_k)^2}\,\Bigl(I_{\{Z_{k-1}\ge U\}}+I_{\{Z_{k-1}\le L\}}\Bigr)
\\= n\sum_{k\ge n}I_{A_{k-1}}\,\frac{B_k}{S_{k-1}+B_k}\,\frac{\abs{H_{k-1}}}{1-Z_{k-1}}
\\\le\frac{2\,c\,n}{1-Z}\,\sum_{k\ge n}I_{A_{k-1}}\,\frac{\abs{H_{k-1}}}{S_{k-1}}\le\frac{4\,c\,n}{m(1-Z)}\,\sum_{k\ge n}\frac{\abs{H_{k-1}}}{{k-1}}
\overset{a.s.}\longrightarrow 0.
\end{gather*}
Hence, to get \eqref{capp8uc4}, it suffices to prove that
\begin{gather*}
n\,\sum_{k\ge n}I_{A_{k-1}}\,\frac{B_k^2}{(S_{k-1}+B_k)^2}\,(X_k-Z_{k-1})^2\overset{a.s.}\longrightarrow\sigma^2.
\end{gather*}
Finally, such condition follows from Lemma \ref{cx3ej876zq1} if $E(V_{n+1}\mid\mathcal{G}_n)\overset{a.s.}\longrightarrow\sigma^2$, where
\begin{gather*}
V_n=n^2I_{A_{n-1}}\,\frac{B_n^2}{(S_{n-1}+B_n)^2}\,(X_n-Z_{n-1})^2.
\end{gather*}
In fact,
\begin{gather*}
E(V_{n+1}\mid\mathcal{G}_n)=I_{A_n}\,(n+1)^2E\Bigl\{\frac{B_{n+1}^2}{(S_n+B_{n+1})^2}\,(X_{n+1}-Z_n)^2\mid\mathcal{G}_n\Bigr\}\le (n+1)^2E\Bigl\{\frac{B_{n+1}^2}{S_n^2}\,(X_{n+1}-Z_n)^2\mid\mathcal{G}_n\Bigr\}
\\=\frac{(n+1)^2}{S_n^2}\,E\bigl\{B_{n+1}^2\,(X_{n+1}-Z_n)^2\mid\mathcal{G}_n\bigr\}=\frac{(n+1)^2}{S_n^2}\,E(B_{n+1}^2)\,E\bigl\{(X_{n+1}-Z_n)^2\mid\mathcal{G}_n\bigr\}
\\=\frac{(n+1)^2}{S_n^2}\,E(B_{n+1}^2)\,Z_n\,(1-Z_n)\overset{a.s.}\longrightarrow\frac{q\,Z\,(1-Z)}{m^2}=\sigma^2.
\end{gather*}
Since the events $A_n$ are eventually true, one similarly obtains
\begin{gather*}
E(V_{n+1}\mid\mathcal{G}_n)\ge I_{A_n}\,(n+1)^2E\Bigl\{\frac{B_{n+1}^2}{(S_n+c)^2}\,(X_{n+1}-Z_n)^2\mid\mathcal{G}_n\Bigr\}
\\=I_{A_n}\,\frac{(n+1)^2}{(S_n+c)^2}\,E\bigl\{B_{n+1}^2\,(X_{n+1}-Z_n)^2\mid\mathcal{G}_n\bigr\}\overset{a.s.}\longrightarrow\sigma^2.
\end{gather*}
Hence, $E(V_{n+1}\mid\mathcal{G}_n)\overset{a.s.}\longrightarrow\sigma^2$. This proves condition \eqref{capp8uc4} and concludes the proof of the lemma.
\end{proof}

Theorem \ref{t2} is a quick consequence of Lemma \ref{pierc3l9}. Define in fact
\begin{gather*}
D_n=\sqrt{n}\,(Z_n-Z)\quad\text{and}\quad F_n=\prod_{i=n}^\infty(1+H_i).
\end{gather*}
Because of Lemma \ref{pierc3l9},
\begin{gather*}
\sqrt{n}\,\bigl(F_n\,Z_n-Z\bigr)=T\,\sqrt{n}\,(W_n-W)\longrightarrow\mathcal{N}(0,\,\sigma^2)
\end{gather*}
conditionally a.s. with respect to $\mathcal{G}$.

If $L<Z<U$ a.s., then $L<Z_n<U$ eventually, which in turn implies $F_n=1$ and $D_n=\sqrt{n}\,\bigl(F_n\,Z_n-Z\bigr)$ eventually. Thus $D_n\longrightarrow\mathcal{N}(0,\,\sigma^2)$, conditionally a.s. with respect to $\mathcal{G}$, provided $L<Z<U$ a.s.

\begin{lem}\label{j6re4}
$P(L<Z<U)=1$.
\end{lem}

\begin{proof}
We just prove $P(Z=L)=0$. The proof of $P(Z=U)=0$ is the same. Since $P(Z=L=0)\le P(Z=0)=0$, it suffices to show that $P(Z=L>0)=0$. Let $H=\{Z=L>0\}$. Toward a contradiction, suppose $P(H)>0$ and define $Q(\cdot)=P(\cdot\mid H)$. Since $\sqrt{n}\,(Z_n-L)$ is $\mathcal{G}_n$-measurable and
\begin{gather*}
D_n=\sqrt{n}\,(Z_n-Z)\le\sqrt{n}\,(Z_nF_n-Z)\quad\text{if }F_n\ge 1,
\end{gather*}
then
\begin{gather*}
I_{\{\sqrt{n}\,(Z_n-L)\le 0\}}=Q\Bigl(D_n\le 0\mid\mathcal{G}_n\Bigr)\ge Q\Bigl(F_n\ge 1,\,\sqrt{n}\,(Z_nF_n-Z)\le 0\mid\mathcal{G}_n\Bigr)
\\\ge Q\Bigl(\sqrt{n}\,(Z_nF_n-Z)\le 0\mid\mathcal{G}_n\Bigr)-Q(F_n<1\mid\mathcal{G}_n)\quad\text{a.s.}
\end{gather*}
Since $Z_n<U$ eventually on $H$, then $F_n\ge 1$ eventually on $H$. Hence, the martingale convergence theorem in \cite{BD} yields $Q(F_n<1\mid\mathcal{G}_n)\overset{Q-a.s.}\longrightarrow 0$. By Lemma \ref{k9cxs34} and $\sigma^2>0$ a.s., it follows that
\begin{gather*}
Q\Bigl(\sqrt{n}\,(Z_nF_n-Z)\le 0\mid\mathcal{G}_n\Bigr)\overset{Q-a.s.}\longrightarrow\mathcal{N}(0,\,\sigma^2)\bigl((-\infty,0]\bigr)=1/2.
\end{gather*}
Thus, $I_{\{\sqrt{n}\,(Z_n-L)\le 0\}}\overset{Q-a.s.}\longrightarrow 1$, namely, $Z_n\le L$ eventually on $H$, which implies the contradiction $Z_n\overset{a.s.}\longrightarrow 1$ on $H$. Thus, $P(H)=0$.
\end{proof}

It remains only to show that $Z$ has non-atomic distribution. This follows from the same argument of Lemma \ref{j6re4}. Suppose in fact $P(Z=z)>0$ for some $z\in (0,1)$ and define $Q(\cdot)=P(\cdot\mid Z=z)$. Then, on the complement of a $Q$-null set, one obtains the contradiction $\sigma^2=q\,z\,(1-z)/m^2>0$ and
\begin{gather*}
\delta_{\sqrt{n}\,(Z_n-z)}(\cdot)=Q\Bigl(D_n\in\cdot\mid\mathcal{G}_n\Bigr)\overset{weakly}\longrightarrow\mathcal{N}(0,\,\sigma^2).
\end{gather*}
This concludes the proof of Theorem \ref{t2}.

\section{Concluding remarks}\label{g6r4e}
Even if conceptually simple, our proofs of Theorems 1-2 are quite long and technical. This is the main reason for assuming $R_n=B_n$. However, in case $L=0$ and $U=1$, such assumption may be weakened into
\begin{gather*}
E(R_n)=E(B_n)\quad\text{for all }n\ge 1;
\end{gather*}
see \cite[Corollary 3]{BCPR}. Also, up to minor complications, various points in the proofs of Theorems 1-2 seem to run under the only assumption that $E(R_n)=E(B_n)$. Thus, we conjecture that Theorems 1-2 are still valid if $R_n=B_n$ is replaced by $E(R_n)=E(B_n)$.

Let
\begin{gather*}
\overline{X}_n=(1/n)\sum_{i=1}^nX_i\quad\text{and}\quad C_n=\sqrt{n}\,(\overline{X}_n-Z_n).
\end{gather*}
If $L=0$ and $U=1$, as shown in \cite[Corollary 3]{BCPR}, one obtains
\begin{gather*}
C_n\longrightarrow\mathcal{N}\bigl(0,\,\sigma^2-Z(1-Z)\bigr)\quad\text{stably.}
\end{gather*}
Thus, one could investigate the asymptotic behavior of $C_n$, or even of the pair $(C_n,D_n)$, in the general case $0\le L<U\le 1$.
Again, this could be (tentatively) performed assuming $E(R_n)=E(B_n)$ instead of $R_n=B_n$.

Another (obvious) improvement is considering multicolor urns rather than 2-colors urns. Indeed, most additional problems arising in the multicolor case are of the notational type.

A last (but minor) development could be generalizing the framework in Subsection \ref{fr6w}. For instance, $B_n\vee R_n\le c$ could be replaced by a suitable moment condition. Or else, $(B_n,R_n)$ independent of $\sigma\bigl(\mathcal{G}_{n-1},\,X_n)$ could be replaced by $(B_n,R_n)$ conditionally independent of $X_n$ given $\mathcal{G}_{n-1}$.


\begin{thebibliography}{99}

\bibitem{AMS} Aletti G., May C., Secchi P. (2009) A central limit theorem, and
related results, for a two-color randomly reinforced urn, {\em Adv. Appl. Probab.}, 41, 829-844.

\bibitem{AGP} Aletti G., Ghiglietti A., Paganoni A.M. (2013) Randomly reinforced urn designs with prespecified allocations, {\em J. Appl. Probab.}, 50, 486-498.

\bibitem{BH} Bai Z.D., Hu F. (2005) Asymptotics in randomized urn models, {\em Ann. Appl. Probab.}, 15, 914-940.

\bibitem{BCPR10} Berti P., Crimaldi I., Pratelli L., Rigo P. (2010)
Central limit theorems for multicolor urns with dominated colors, {\em Stoch. Proc. Appl.}, 120, 1473-1491.

\bibitem{BCPR}Berti P., Crimaldi I., Pratelli L., Rigo P. (2011)
A central limit theorem and its applications to multicolor randomly
reinforced urns, {\em J. Appl. Probab.}, 48, 527-546.

\bibitem{BD} Blackwell D., Dubins L.E. (1962) Merging of opinions with increasing information, {\em Ann. Math. Statist.}, 33, 882-886.

\bibitem{CPS} Chauvin B., Pouyanne N., Sahnoun R. (2011) Limit distributions for large Polya urns, {\em Ann. Appl. Probab.}, 21, 1-32.

\bibitem{CLP}Crimaldi I., Letta G., Pratelli L. (2007)
A strong form of stable convergence, {\em Sem. de Probab. XL}, LNM,
1899, 203-225.

\bibitem{CRIM} Crimaldi I. (2009) An almost sure conditional convergence result and an application to a
generalized Polya urn, {\em Internat. Math. Forum}, 4, 1139-1156.

\bibitem{LP} Laruelle S., Pages G. (2013) Randomized urn models revisited using stochastic
approximation, {\em Ann. Appl. Probab.}, 23, 1409-1436.

\bibitem{MAH} Mahmoud H. (2008) {\em Polya urn models}, Chapman-Hall, Boca Raton, Florida.

\bibitem{MF} May C., Flournoy N. (2009) Asymptotics in response-adaptive designs generated by
a two-color, randomly reinforced urn, {\em Ann. Statist.}, 37,
1058-1078.

\bibitem{ZHANG} Zhang L. (2014) A Gaussian process approximation for two-color randomly reinforced urns, {\em Electron. J. Probab.}, 19, 1-19.

\bibitem{ZHCC} Zhang L., Hu F., Cheung S.H., Chan W.S. (2014) Asymptotic properties of multicolor randomly reinforced Polya urns, {\em Adv. Appl. Probab.}, 46, 585-602.

\end{thebibliography}
\end{document}